\documentclass[11pt]{amsart}
\usepackage{srcltx} 
\usepackage{latexsym} 
\usepackage[titletoc]{appendix}
\usepackage{lineno}
\usepackage{bm}
\usepackage{enumerate}
\usepackage{esint}
\usepackage{mathtools}
\usepackage[plainpages=true,pdfpagelabels,hypertexnames=true,colorlinks=true,pdfstartview=FitV,linkcolor=black,citecolor=black,urlcolor=black]{hyperref}
\usepackage[ddmmyyyy]{datetime}
\usepackage{etoolbox}
\usepackage{twoopt}
\usepackage{color}
\usepackage{amsfonts,amssymb,amscd,amsmath}
\usepackage{xparse}

\DeclareMathOperator{\dv}{div}
\DeclareMathOperator{\capacity}{Cap}
\DeclareMathOperator{\mm}{M}
\DeclareMathOperator{\ee}{\mathbb{E}}

\newcommand{\de}{\delta}
\newcommand{\ep}{\varepsilon}
\newcommand{\ga}{\gamma}

\newcommand{\Om}{\Omega}

\newcommand{\RR}{\mathbb{R}}
\newcommand{\itl}[1][\Om]{\int_{#1}}
\def\aa{\mathcal{A}}

\newcommand{\V}{W_0^{1,p}(\Om)}

\newcommand{\axgrad}[1]{\mathcal{A}(x,\ifblank{#1}{\nabla u\:}{#1})}
\newcommand{\dx}[1][x]{\, d#1} 
\newcommand{\cpt}[1][{1,p}]{\capacity_{#1}}
\newcommand{\trm}[1]{\quad \textrm{#1}\quad }

\newcommand{\mps}[1][\Om]{\mm^{1,p}(#1)}

\newcommand{\wpo}[1][\Om]{W_0^{1,p}(#1)}

\newcommand{\abs}[1]{\lvert #1 \rvert}
\newcommand{\expt}[1][\abs{T_j (u_n) }]{e^{\de #1}}
\newcommand{\exptk}[1][\abs{T_j (u_k) }]{e^{\mu_0 #1}}

\def\gbmo{$(\ga, \, R_0)$-BMO }
\def\gflt{$(\ga, \, R_0)$-Reifenberg flat }

\def\dv{\mathop{\rm div}}

\def\bea{\begin{equation}\begin{aligned}}
\def\ena{\end{aligned}\end{equation}}

\def\beas{\begin{equation*}\begin{aligned}}
\def\enas{\end{aligned}\end{equation*}}

\def\integral{\int}
\def\fintegral{\fint}

\def\nc{\newcommand}

\nc\m[1]{\left| #1\right|}
\nc\norm[1]{\left\|#1\right\|}
\NewDocumentCommand{\ceil}{s O{} m}{%
  \IfBooleanTF{#1} 
    {\left\lceil#3\right\rceil} 
    {#2\lceil#3#2\rceil} 
}

\newtheorem{theorem}{Theorem}[section]
\newtheorem{lemma}[theorem]{Lemma}

\newtheorem{definition}[theorem]{Definition}

\newtheorem{hypT}[theorem]{Hypothesis}

\newtheorem{assmT}[theorem]{Assumption}

\newtheorem{remark}[theorem]{Remark}        
\numberwithin{equation}{section}

\begin{document}
\title[Quasilinear equations with natural growth and Sobolev multipliers]{Quasilinear equations with natural growth in the gradients in spaces of Sobolev multipliers}

\author[Karthik Adimurthi]
{Karthik Adimurthi$^{1}$}
\address{${}^1$ Department of Mathematical Sciences, 
Seoul National University, GwanAkRo 1, Gwanak-Gu, 
Seoul 08826, 
South Korea.}
\email{kadimurthi@snu.ac.kr \and karthikaditi@gmail.com}

\author[Nguyen Cong Phuc]
{Nguyen Cong Phuc$^{2}$}
\address{${}^2$ Department of Mathematics,
Louisiana State University,
303 Lockett Hall, Baton Rouge, LA 70803, USA.}
\email{pcnguyen@math.lsu.edu}

\thanks{$^{1}$ Supported in part by National Research Foundation of Korea grant funded by the Korean government (MEST) (NRF-2015R1A2A1A15053024)}
\thanks{$^{2}$ Supported in part by Simons Foundation, award number 426071}

\begin{abstract}We study the  existence problem for a class of nonlinear elliptic equations whose prototype is of the form  $-\Delta_p u = |\nabla u|^p + \sigma$ in a bounded domain $\Om\subset \RR^n$. Here $\Delta_p$, $p>1$, is the standard $p$-Laplacian operator defined by $\Delta_p u={\rm div}\, (|\nabla u|^{p-2}\nabla u)$, and 
the datum $\sigma$ is a signed distribution in $\Om$.  The class of solutions that we are interested in consists  of functions $u\in W^{1,p}_0(\Om)$ such that 
$|\nabla u|\in M(W^{1,p}(\Om)\rightarrow L^p(\Om))$, a space pointwise Sobolev  multipliers  consisting of functions $f\in L^{p}(\Om)$ such that   
\begin{equation*}
 \int_{\Om} |f|^{p} |\varphi|^p dx \leq C \int_{\Om} (|\nabla \varphi|^p + |\varphi|^p)  dx \quad \forall \varphi\in C^\infty(\Om),
\end{equation*}
for some $C>0$.   This is a natural class of solutions at least when  the distribution $\sigma$ is nonnegative and compactly supported in $\Om$. We show essentially  that, 
 with only a gap in the smallness constants, the above equation has a solution in this class if and only if one can write $\sigma={\rm div}\, F$ for a vector field $F$ such
that $|F|^{\frac{1}{p-1}}\in M(W^{1,p}(\Om)\rightarrow L^p(\Om))$.

As an important application, via the exponential transformation $u\mapsto v=e^{\frac{u}{p-1}}$, we obtain an existence result for  the quasilinear  equation of Schr\"odinger type $-\Delta_p v = \sigma\, v^{p-1}$, $v\geq 0$ in $\Om$, and $v=1$ on $\partial\Om$, which is interesting in its own right.

\end{abstract}
	
\maketitle

\section{Introduction}\label{Sec1}
In this work, we study the existence problem for the quasilinear elliptic  equation
\begin{equation}
 \label{basic_pde2}
\left\{ \begin{array}{ll}
-\dv \aa(x, \nabla u) = \mathcal{B}(x, u, \nabla u) + \sigma & \trm{in} \Omega, \\
u = 0 & \trm{on} \partial \Omega, 
\end{array}
\right.
\end{equation}
in a bounded domain $\Omega \subset \RR^n$, where the principal operator $\dv \aa(x, u, \nabla u)$ is a Leray-Lions operator defined on $W_0^{1,p}(\Om)$ and $|\mathcal{B}(x, u, \nabla u)| \lesssim |\nabla u|^p$, $p>1$.
Precise  assumptions on the domain $\Om$ and the nonlinearities $\aa$, $B$ will be made explicitly later.  Here  the `datum' $\sigma$ is  a general  distribution in $\Om$, and
$\wpo$ is defined as the completion of $C_c^\infty(\Om)$ under the semi-norm $\norm{\nabla (\cdot)}_{L^{p}(\Om)}$.

 A typical example of \eqref{basic_pde2} after which it is modeled is the following quasilinear elliptic  equations with gradient nonlinearity of natural growth  of the form
\begin{equation} \label{basic_pde}
-\Delta_p u = |\nabla u|^p + \sigma \text{ in } \Omega,  \qquad u = 0 \text{ on } \partial \Omega, 
\end{equation}
where  $\Delta_p u:= \dv (|\nabla u|^{p-2} \nabla u)$, $p>1$, is the $p$-Laplacian operator.

When $p=2$, equation \eqref{basic_pde} becomes a  stationary  viscous Hamilton-Jacobi equation, also known as the Kardar-Parisi-Zhang equation that appears  in the physical theory  of growth and roughening of surfaces  \cite{KPZ, KS}. Moreover, via the transformation $u\mapsto v:=e^{\frac{u}{p-1}}$, this equation can be transformed into the Schr\"odinger type equation 
\begin{equation*}
-\Delta_p v =   (p-1)^{1-p}\sigma\, v^{p-1} \text{ in } \Omega, \qquad v \geq 0  \text{ in } \Omega, \qquad v = 1   \text{ on } \partial \Omega, 
\end{equation*}
a connection that we shall discuss at the end of this section.

When it comes to the existence theory,   it is well-known that in order for \eqref{basic_pde} to have a solution the datum $\sigma$ must be both small and regular enough. For example, if  $\sigma$ is a nonnegative  locally finite measure in $\Om$ and  the first equation in \eqref{basic_pde}  has a $W^{1,p}_{\rm loc}(\Om)$ solution (without any boundary condition), then $\sigma$ must obey the weighted Poincar\'e-Sobolev inequality (see \cite{HMV, JMV1, JMV2}):
\begin{equation}\label{possig}
\int_{\Om} |\varphi|^p d\sigma \leq  (p-1)^{p-1} \int_\Om |\nabla \varphi|^p dx \quad \forall \varphi\in C_c^\infty(\Omega).
\end{equation}

Moreover, when $\sigma\geq 0$  the nonlinear term  $|\nabla u|^p$ also obeys a similar weighted inequality
\begin{equation}\label{nablau-cond}
\int_{\Om} |\varphi|^p |\nabla u|^pdx  \leq p^p \int_\Om |\nabla \varphi|^p dx \quad \forall \varphi\in C_c^\infty(\Omega).
\end{equation}

If we assume in addition that ${\rm supp}(\sigma)=K$ where $K$ is a compact set in $\Om$, then by multiplying by a  cutoff function $\chi\in C_c^\infty(\Om)$, $0\leq \chi\leq 1$, and $\chi=1$ on $K$, we  see from 
\eqref{possig} that 
\begin{equation}\label{possig-Rn}
\int_{\Om} |\varphi|^p d\sigma \leq  \lambda \int_{\RR^n} (|\nabla \varphi|^p +|\varphi|^p) dx \quad \forall \varphi\in C_c^\infty(\RR^n),
\end{equation}
with a constant $\lambda>0$. \emph{Note that the `test functions' $\varphi$ in \eqref{possig-Rn} are now allowed to have support not contained in $\Om$. However,  in general from 
\eqref{nablau-cond} we cannot say that $|\nabla u|^p$ obeys the similar inequality
\begin{equation}\label{nablau-cond-Rn}
\int_{\Om} |\varphi|^p |\nabla u|^pdx  \leq A \int_{\RR^n} (|\nabla \varphi|^p + |\varphi|^p)  dx \quad \forall \varphi\in C_c^\infty(\RR^n),
\end{equation}
for some $A>0$, not even when $u\in W^{1,p}_0(\Om)$.} The main difference between  \eqref{nablau-cond} and \eqref{nablau-cond-Rn}
lies in the behavior of $|\nabla u|^p$ near the boundary of $\Om$. \emph{In contrast to \eqref{nablau-cond}, inequality \eqref{nablau-cond-Rn} requires that $|\nabla u|^p$ have stronger regularity up to  the boundary of $\Om$.}

In this paper, we only insist on obtaining solutions to  \eqref{basic_pde2} that belong to the class $\mathcal{C}$ of functions 
$u\in W^{1,p}_0(\Om)$ such that inequality \eqref{nablau-cond-Rn} holds with some $A>0$. Our goal is to find the largest 
space $\mathcal{F}$ of data on $\Om$ so that whenever $\sigma\in \mathcal{F}$ with a sufficiently small 
norm then    \eqref{basic_pde2} has a solution in $\mathcal{C}$. In brief, our main result states  that, with only a gap in the smallness constants, equation
\eqref{basic_pde2} has a solution in the class $\mathcal{C}$ if and only if the distribution $\sigma$ can be written 
in the form $\sigma={\rm div}\, F$ for a vector field $F\in L^{\frac{p}{p-1}}(\Om,\RR^n)$ such that 
\begin{equation}\label{Vec-cond}
 \int_{\Om} |F|^{\frac{p}{p-1}} |\varphi|^p dx \leq \lambda \int_{\RR^n} (|\nabla \varphi|^p + |\varphi|^p)  dx \quad \forall \varphi\in C_c^\infty(\RR^n),
\end{equation}
for some $\lambda>0$. 
For the simpler equation \eqref{basic_pde} on, say, $C^1$ domains our results read as follows.

\begin{theorem}\label{weakzero} {\rm (i)}  Suppose that \eqref{basic_pde} has a solution in $u\in W^{1,p}_0(\Om)$ such that 
\eqref{nablau-cond-Rn} holds for some $A>0$ then  it is necessary that  $\sigma= {\rm div}\, F$
for a vector field $F\in L^{\frac{p}{p-1}}(\Om,\RR^n)$ such that \eqref{Vec-cond} holds with a $\lambda>0$. 

\noindent {\rm (ii)} Conversely, suppose that $\Om$ is a bounded $C^1$ domain. Then  there 
exists a constant $\lambda_0=\lambda_0(n,p,\Om)>0$ such that if $\sigma={\rm div}\, F$ for a vector field $F$ satisfying \eqref{Vec-cond}  with some   $\lambda\in (0, \lambda_0]$, then  \eqref{basic_pde} has a  solution  $u\in W^{1,p}_{0}(\Om)$ satisfying the weighted inequality \eqref{nablau-cond-Rn} for some $A>0$. 
\end{theorem}

The condition \eqref{Vec-cond} simply means that the function $|F|^\frac{1}{p-1}\chi_{\Om}$ belongs to the space of Sobolev multipliers $M(W^{1,p}(\RR^n)\rightarrow L^p(\RR^n))$, which consists of functions $f\in L^{p}_{\rm loc}(\RR^n)$ such that   
\begin{equation*}
 \int_{\RR^n} |f|^{p} |\varphi|^p dx \leq C \int_{\RR^n} (|\nabla \varphi|^p + |\varphi|^p)  dx \quad \forall \varphi\in C_c^\infty(\RR^n),
\end{equation*}
for some $C>0$. The norm of such $f$ is the $p$-th root of the best constant $C$ in the above inequality.

For our purpose, we denote by ${\rm M}^{1,p}(\Om)$ the space of functions $f\in L^p(\Om)$ such that $f\chi_\Om\in M(W^{1,p}(\RR^n)\rightarrow L^p(\RR^n))$ with norm $\norm{f}_{{\rm M}^{1,p}(\Om)}:= \norm{f\chi_{\Om}}_{M(W^{1,p}(\RR^n)\rightarrow L^p(\RR^n))}$. \textcolor{black}{The space  ${\rm M}^{1,p}(\Om)$  can also be described using the capacity associated to the  Sobolev space $W^{1,p}(\RR^n)$; see Section \ref{ABO} below.}
\textcolor{black}{Moreover, it is known that} for any bounded Lipschitz domain $\Om$, the space  ${\rm M}^{1,p}(\Om)$ coincides with the  multiplier space  $M(W^{1,p}(\Om)\rightarrow L^p(\Om))$, which consists of functions $f\in L^{p}(\Om)$ such that   
\begin{equation*}
 \int_{\Om} |f|^{p} |\varphi|^p dx \leq C \int_{\Om} (|\nabla \varphi|^p + |\varphi|^p)  dx \quad \forall \varphi\in C^\infty(\Om),
\end{equation*}
for some $C>0$; see \cite[Theorem 9.3.1]{MS2}.

 It is worth pointing out, as we show in Theorem \ref{main2} below, that the solution  $u$ obtained in Theorem \ref{weakzero}(ii) obeys a stability estimate
$$\norm{|\nabla u|}_{{\rm M}^{1,p}(\Om)} \leq T_0  \norm{|F|^{\frac{1}{p-1}}}_{{\rm M}^{1,p}(\Om)}$$
for a constant $T_0>0$. Moreover, we have $e^{\mu|u|}-1\in W^{1,p}_{0}(\Om)$ provided $\mu\in[0, \mu_0]$ where $\mu_0=C \norm{|F|^{\frac{1}{p-1}}}_{{\rm M}^{1,p}(\Om)}^{-1}$ for some $C>0$. In particular, the solution here is zero whenever $\sigma=0$.
Note that, even for $\sigma=0$, in general  $W^{1,p}_0(\Om)$ solutions to \eqref{basic_pde} are not unique;  
see \cite[Remark 2.11]{FM2}.
 
 An existence criterion in the spirit of Theorem \ref{weakzero} also holds for equation $\eqref{basic_pde2}$
under quite general assumptions on $\mathcal{A}$, $\mathcal{B}$ and $\Om$. In particular,  $\mathcal{A}(x,\xi)$ could be discontinuous in the $x$-variable and $B$ could include a zero order term. Moreover, $\Om$ could be irregular and include
certain Lipschitz or fractal domains. These assumptions will be made precise in the next section. The result for equation $\eqref{basic_pde2}$, which  is the main result of the paper and includes Theorem \ref{weakzero} as a special case,   will be  treated in Section \ref{main-AB} (see Theorem \ref{main2} below).

We   next have the following  remarks.
\begin{remark} Let $\widetilde{\Om}$ be another bounded open set such that $\Om\Subset \widetilde{\Om}$. Then by Poincar\'e's inequality we see that 
 \eqref{Vec-cond} is also  equivalent to the homogeneous inequality 
\begin{equation*}
 \int_{\Om} |F|^{\frac{p}{p-1}} |\varphi|^p dx \leq \lambda \int_{\widetilde{\Om}} |\nabla \varphi|^p   dx \quad \forall \varphi\in C_c^\infty(\widetilde{\Om}),
\end{equation*}
for some $\lambda>0$. One also has similar statements for \eqref{possig-Rn} and \eqref{nablau-cond-Rn}. 
\end{remark}

\begin{remark} Let $\sigma$ be a nonnegative measure such that ${\rm supp}(\sigma)\Subset\Om$. If the first equation in \eqref{basic_pde}  has a distributional solution $u\in W^{1,p}_{\rm loc}(\Om)$ (without any boundary condition), then it is still necessary that $\sigma={\rm div}\, F$ with $|F|^{\frac{1}{p-1}}\in {\rm M}^{1,p}(\Om)$. This follows from \eqref{possig-Rn} and Lemma \ref{mutoF0} below. This shows that the space  ${\rm M}^{1,p}(\Om)$ is quite natural for \eqref{basic_pde} at least for such data $\sigma$.
\end{remark}

We now  mention some of the relevant results in the literature on the existence of $W^{1,p}_0(\Om)$ solutions to   \eqref{basic_pde} or \eqref{basic_pde2}. 
 In \cite{FM1, FM2} an existence result in $W^{1,p}_0$ was obtained for small data $\sigma\in [W_0^{1, \frac{n}{n-p+1}}(\Om)]^{*}$, i.e., $\sigma=
\dv F$ where $|F|^{\frac{1}{p-1}}\in L^{n}(\Om)$ with a small norm. Later, it was shown in \cite{FM3} that if $\sigma=
\dv F$ where $|F|^{\frac{1}{p-1}}\in L^{n, \infty}(\Om)$ (the weak Lebesgue space) with a small norm than \eqref{basic_pde2} admits a solution. Recently in \cite{MP}, an existence result was obtained for  $\sigma=\dv F$ provided $|F|^{\frac{1}{p-1}}$ is small in $\mathcal{L}^{(1+\varepsilon)p, (1+\varepsilon)p}(\Om)$ provided $0<\varepsilon < n/p-1$. Here $\mathcal{L}^{(1+\varepsilon)p, (1+\varepsilon)p}(\Om)$, $0 < \varepsilon < n/p-1$, is a Morrey space with norm given by 
$$\norm{f}^{(1+\varepsilon)p}_{\mathcal{L}^{(1+\varepsilon)p, (1+\varepsilon)p}(\Om)}= \sup \Big[r^{(1+\varepsilon)p-n} \int_{B_r(z)\cap\Om} |f|^{(1+\varepsilon)p}dx\Big], $$
where the supremum is taken over $z\in\Om$ and $0<r\leq {\rm diam}(\Om)$. Note that one has the following inclusions:
$$L^{n}(\Om) \subset L^{n, \infty}(\Om) \subset \mathcal{L}^{(1+\varepsilon)p, (1+\varepsilon)p}(\Om)$$
provided $0< \varepsilon<n/p-1$.  That $|F|^{\frac{1}{p-1}} \in \mathcal{L}^{(1+\varepsilon)p, (1+\varepsilon)p}(\Om)$ implies inequality  \eqref{Vec-cond}, i.e.,
$$\mathcal{L}^{(1+\varepsilon)p, (1+\varepsilon)p}(\Om) \subset {\rm M}^{1,p}(\Om), \quad 0<\varepsilon < n/p-1,$$
 is well-known as it is a special case of the so-called Fefferman-Phong type conditions (see, e.g.,  \cite{ChWW,DPT,  Fef, P, SW}).

We note that there are also existence results obtained for \eqref{basic_pde} under weaker conditions on 
$\sigma$ and sometimes with sharp  constants of smallness; see \cite{ADP, FV, HBV} for nonnegative measure data and \cite{AP2,  JMV1, JMV2} for distributional data. See also \cite{FM1, FM2, FM3}. However, the solutions obtained in those papers may not  behave very well at the boundary of $\Om$, i.e., in general they do not satisfy  inequality \eqref{nablau-cond-Rn}. See also the 
earlier work \cite{HMV} where an existence result was obtained in the whole space $\Om=\RR^n$ in the  `linear' case $p=2$ for nonnegative measure data.

We now briefly describe the strategy that we use to construct a solution $u\in W^{1,p}_0(\Om)$ of \eqref{basic_pde} such that 
$|\nabla u|\in {\rm M}^{1,p}(\Om)$ under the assumption  $\sigma={\rm div}\, F$ where $|F|^{\frac{1}{p-1}}\in {\rm M}^{1,p}(\Om)$  with a small norm. As in \cite{FM2}, we start with  the approximate equation
\begin{equation} \label{basic_pde-k}
-\Delta_p u = \frac{|\nabla u|^p}{1+ k^{-1}|\nabla u|^p}  + \sigma \text{ in } \Omega,  \qquad u = 0 \text{ on } \partial \Omega, 
\end{equation}
 where the parameter $k>0$ is to be sent to infinity eventually.   Since $\sigma\in (\wpo)^{*}$ and  the first term on the right-hand side is uniformly bounded, by the theory of pseudomonotone operators (see, e.g., \cite{Li}), there exists a solution $u_{k}\in \wpo$ to \eqref{basic_pde-k}. However, this solution may not satisfy the property 
that $|\nabla u_k|\in {\rm M}^{1,p}(\Om)$. Thus, to have this requiblack property for  $u_k$ we have to construct it by a different way. As $\norm{|F|^{\frac{1}{p-1}}}_{{\rm M}^{1,p}(\Om)}$ is small, it is natural to use Schauder's Fixed Point Theorem in a small ball of    ${\rm M}^{1,p}(\Om)$. The main difficulty in this approach is an a priori   gradient estimate
of the form 
$$\norm{|\nabla u|}_{{\rm M}^{1,p}(\Om)}\leq C \norm{|F|^{\frac{1}{p-1}}}_{{\rm M}^{1,p}(\Om)}$$
for solutions $u$ to  the basic equation
\begin{equation}\label{linear_pde} 
-\Delta_p u =  {\rm div}\, F \text{ in } \Omega,  \qquad u = 0 \text{ on } \partial \Omega. 
\end{equation}
Such a delicate  \textcolor{black}{gradient estimate} can be obtained from an end-point weighted gradient estimate for \eqref{linear_pde} and has been prepablack  in our earlier work \cite{AP1}; see Lemma \ref{capboundu} below.     Once  solutions $\{u_k\}$ to
\eqref{basic_pde-k} have been obtained with gradients being uniformly controlled in ${\rm M}^{1,p}(\Om)$, the next step is to pass to the limit in \eqref{basic_pde-k}, with $u_k$ in place of $u$, as $k\rightarrow\infty$.  For that, it is enough to show the strong convergence of $\{u_k\}$ in $W^{1,p}_0(\Om)$, a task that can be done via the truncation technique
and appropriate test functions as in  \cite{FM2, FM3}. We mention that in our scenario this is possible since  we have a uniform bound for $\{e^{|u_k|}-1 \}$ in $W^{1,p}_0(\Om)$, another important a priori estimate also obtained in Lemma \ref{capboundu}.

To conclude this section, we discuss a connection of \eqref{basic_pde} and
a Schr\"odinger  type equation with distributional potential:
\begin{equation}\label{basic_pde-schr}
-\Delta_p v =   \sigma\, v^{p-1} \text{ in } \Omega, \qquad v \geq 0  \text{ in } \Omega, \qquad v = 1   \text{ on } \partial \Omega. 
\end{equation}

 This equation  is interesting in its own right and its existence theory has been studied,  e.g., in \cite{ADP, HBV, JMV1, JMV2}. For $\sigma\in (W^{1,p}_0(\Om))^{*}$, by a solution of \eqref{basic_pde-schr} we mean a nonnegative function $v$ such  that  $v-1\in W^{1,p}_{0}(\Om)$,  $v^{p-1}\in W^{1,p}_{\rm loc}(\Om)$, and 
\begin{equation*}
\int_{\Omega} |\nabla v|^{p-2} \nabla v \cdot \nabla \varphi dx=  \langle \sigma, v^{p-1}\varphi\rangle \qquad \forall  \varphi\in C_c^\infty(\Om). 
\end{equation*}
Since $v^{p-1}\varphi\in W^{1,p}(\Om)$ and ${\rm supp}(v^{p-1}\varphi)\Subset\Om$, this definition makes sense even for 
$\sigma$ such that $\sigma\in (W^{1,p}_0(\Om'))^{*}$ for any open set $\Om'\Subset\Om$ (see \cite{JMV1, JMV2}).

Formally, by using the transformation  $u\mapsto v:=e^{\frac{u}{p-1}}$, the equation 
\begin{equation} \label{basic_pde3}
-\Delta_p u = |\nabla u|^p + (p-1)^{p-1}\sigma  \text{ in } \Omega, \qquad u = 0  \text{ on } \partial \Omega, 
\end{equation}
is transformed into  \eqref{basic_pde-schr}.  Indeed,  using $\phi_k:=\varphi \min\{ e^u, k\}$, $k>0, \varphi\in C_c^\infty(\Om)$, as a test function for \eqref{basic_pde3} and then letting $k\rightarrow\infty$ one can rigorously show from Theorem \ref{main2}(ii) the following existence result for  \eqref{basic_pde-schr}.

\begin{theorem}  Let  $\Om$ be a bounded $C^1$ domain in $\RR^n$. There 
exist $\lambda_1=\lambda_1(n,p,\Om)>0$ and $T_1=T_1(n,p,\Om)>0$ such that if $\sigma={\rm div}\, F$ with $\norm{|F|^{\frac{1}{p-1}}}_{{\rm M}^{1,p}(\Om)}\leq \lambda_1^{1/p}$ then \eqref{basic_pde-schr} has a nonnegative  solution 
$v$ with   $|\nabla \log(v)|\in {\rm M}^{1,p}(\Om)$ and
$$\norm{|\nabla \log(v)|}_{{\rm M}^{1,p}(\Om)} \leq T_1  \norm{|F|^{\frac{1}{p-1}}}_{{\rm M}^{1,p}(\Om)}.$$
Moreover, we have $v^{q}-1\in W^{1,p}_0(\Om)$ for all $q\in[0,\mu_1]$ where 
$$\mu_1=C_1(n,p,\Om) \norm{|F|^{\frac{1}{p-1}}}_{{\rm M}^{1,p}(\Om)}^{-1}\geq \max\{1, p-1\}.$$ 
\end{theorem}

\section{Assumptions on \texorpdfstring{$\mathcal{A}$}., \texorpdfstring{$\mathcal{B}$}., \texorpdfstring{$\Om$}., and  preliminary results} \label{ABO}

We now make precise the assumptions on the nonlinearities $\mathcal{A}$,  $\mathcal{B}$, and on the domain $\Om$ that appear in equation \eqref{basic_pde2}. All of these assumptions will be needed in Theorem \ref{main2}(ii) below.

\noindent{\bf Assumption 1.} In \eqref{basic_pde2}, the nonlinearity $\aa : \RR^n \times \RR^n \rightarrow \RR^n$ is a  Carath\'edory  function, i.e., $\aa(x,\xi)$ is measurable in $x$ for every $\xi$ and continuous in $\xi$ for a.e. $x\in\RR^n$. Moreover, 
$\mathcal{A}(x, \xi)$ is continuously differentiable in $\xi$ away from the origin for a.e. $x \in \RR^n$. We assume that for some $p>1$, it holds that
\begin{gather}
\label{cond2} \langle \aa(x,\xi) - \aa(x,\eta), \xi - \eta \rangle \geq \Lambda_0 (|\xi|^2 + |\eta|^2)^{\frac{p-2}{2}} |\xi- \eta|^2,\\
\label{cond1} |\aa(x,\xi)| \leq \Lambda_1 |\xi|^{p-1}, \quad |\nabla_\xi \aa(x,\zeta)| \leq \Lambda_1 |\xi|^{p-2}
\end{gather}
for every  $(\xi, \eta)\in \RR^n \times \RR^n\setminus (0,0)$ and a.e. $x \in\RR^n$. Here $\Lambda_0$ and $\Lambda_1$ are positive constants.

Additionally, we suppose that $\aa(x,\xi)$ satisfies the following \gbmo condition in the $x$-variable, where $\gamma>0$ is sufficiently small.

\begin{definition}
\label{BMO-condition}
 Given two positive numbers $\ga$ and $R_0$, we say that $\aa(x,\xi)$ satisfies a \gbmo condition  if
\beas
\left[ \aa \right]^{R_0} := \sup\limits_{y \in \RR^n,\, 0<r\leq R_0}  \fintegral_{B_r(y)} \Upsilon(\aa,\, B_r(y)))(x) \, dx  \leq \ga, 
\enas
where for a ball $B$ we set
\beas
\Upsilon(\aa,B) (x) := \sup_{\xi \in \mathbb{R}^{n}\setminus \{0\}} \frac{\left|\aa({x, \xi}) - \frac{1}{|B|} \integral_B \aa(y,\xi)\, dy\right|}{|\xi|^{p-1}}.
\enas
\end{definition}

Note that in the linear case, where $\aa(x,\xi) = A(x)\xi$ for an elliptic matrix $A(x)$, we see that
\beas
\Upsilon(\aa,B)(x) \leq \left|A(x) - \frac{1}{|B|} \integral_B A(y)\, dy\right|
\enas
for a.e. $x \in \RR^n$. Thus Definition \ref{BMO-condition} can be viewed as a natural extension of the standard small BMO condition to the nonlinear setting.  We remark that the \gbmo condition allows the nonlinearity $\aa(x,\xi)$ to have certain discontinuity in $x$, and it can be used as an appropriate substitute for the Sarason VMO  (vanishing mean oscillation) condition  \cite{Sa}.

\noindent{\bf Assumption 2.} In \eqref{basic_pde2}, the nonlinearity $\mathcal{B} : \Om \times \RR \times \RR^n \rightarrow \RR$ is a Carath\'edory  function which satisfies, for a.e. $x\in\Om$, every $s\in\RR$, and every $\xi\in\RR^n$,
\begin{equation}\label{Bcond}
|\mathcal{B}(x,s,\xi)|\leq b_0 |\xi|^p +b_1 |s|^m, \quad \mathcal{B}(x,s,\xi){\rm sign}(s) \leq  b_2 |\xi|^p,
\end{equation}
 where $m>p-1$, and $b_0, b_1$, $b_2$ are nonnegative constants. 

\noindent{\bf Assumption 3.} With regard to the underlying domain $\Om$, we assume that its boundary  is sufficiently  flat in the sense of Reifenberg. This means essentially that at each boundary point and every scale, the boundary of $\Om$ is trapped  in between two hyperplanes separated by a distance proportional to the scale. Precisely, we assume that 
$\Om$ is \gflt  for a sufficiently small $\gamma>0$. Below is the definition of a \gflt domain.
\begin{definition}
Given $\ga \in (0,1)$ and $R_0 > 0$, we say that $\Omega$ is  \gflt  if for every $x_0 \in \partial \Omega$ and every $r \in (0,R_0]$, there exists a system of coordinates $\{y_1,y_2,\ldots,y_n\}$, which may depend on $r$ and $x_0$, so that in this coordinate system $x_0 = 0$ and that 
\beas
B_r(0) \cap \{y_n > \ga r\} \subset B_r (0) \cap \Omega \subset B_r(0) \cap \{y_n > -\ga r\}.
\enas
\end{definition}

For more on Reifenberg flat domains and their many applications, we refer to the papers \cite{HM,Jon,KT1,KT2,Rei,Tor}. We mention here that Reifenberg flat domains can be very rough. They include Lipschitz domains with sufficiently small Lipschitz constants (see \cite{Tor}) and even some domains with fractal boundaries. In particular, all bounded  domains 
with $C^1$ boundaries are allowed in this paper.

Let ${\rm\bf G}_1 \mu$ be the first order Bessel's potential of a nonnegative locally finite measure $\mu$ defined by
$${\rm\bf G}_1 \mu(x)=\int_{\RR^n} G_1(x-y)  d\mu(y), \qquad x\in\RR^n,$$   
where $G_1(x)$ is the Bessel kernel of order one defined via its Fourier transform by $\hat{G}_1(\xi)= (1+|\xi|^2)^{-1/2}$.

{Let  $\cpt(\cdot)$ denote the capacity 
associated to the Sobolev space $W^{1,p}(\RR^n)$, i.e., 
$$\cpt(K) := \inf \left\{ \int_{\RR^n}(|\nabla \phi|^p  + \varphi^p) dx\, : \ \phi \geq 1 \text{ on } K, \ \ \,  0 \leq \phi \leq 1,\ \ \phi \in C_c^{\infty} (\RR^n) \right\}$$
\textcolor{black}{for each compact set $K\subset\RR^n$.}  It is well-known that $\cpt(\cdot)$ is equivalent to the Bessel capacity 
$${\rm Cap}_{{\rm\bf G}_1 ,p}(K):= \inf\left\{\int_{\RR^n} f^p dx:  f\in L^p(\RR^n), f\geq 0, {\rm ~and~} {\rm\bf G}_1 f \geq 1 {\rm ~on~} K \right\}.$$


We next recall a special case of  Theorem 1.2 in \cite{MV}. This theorem enables us to reformulate the existence problem for \eqref{basic_pde2} by means of the  capacity   $\cpt(\cdot)$.

\begin{theorem}\label{CHAR} Let  $\nu$ be  a nonnegative locally finite
measure in $\RR^n$. Then  the following properties of $\nu$
are equivalent.

{\rm (i)} There is a constant $A_1>0$ such that
\begin{equation*}
\int_{\RR^n} |\varphi|^p d\nu \leq A_1 \int_{\RR^n} (|\nabla \varphi|^{p}) + |\varphi|^p)dx
\end{equation*}
for all  $\varphi\in C_c^{\infty}(\RR^n)$.

{\rm (ii)} There is a constant $A_2>0$ such that
\begin{equation*}
\int_{\RR^n} ({\rm\bf G}_{1}f)^{p}d\nu \leq A_2 \int_{\RR^n} f^{p}dx
\end{equation*}
for all nonnegative $f\in L^{p}(\RR^n)$.

{\rm (iii)} There is a constant $A_3>0$ such that
\begin{equation*} 
\nu(K)\leq A_3\, {\rm Cap}_{1, \, p}(K)
\end{equation*}
for all compact sets $K \subset \RR^n$.

{\rm (iv)} There is a constant $A_4>0$ such that

\begin{equation*} 
\int_{K} ({\rm\bf G}_{1}\nu)^{\frac{p}{p-1}}dx   \leq A^{\frac{p}{p-1}}_4\, {\rm Cap}_{1, \, p}(K)
\end{equation*}
for all compact sets $K\subset\RR^n$.

Moreover, the least possible values of the constants $A_1$, $A_2$, $A_3$, and $A_4$ are  comparable to each other. 
\end{theorem}

We now introduce a function space associated to the capacity   $\cpt(\cdot)$, which  plays a crucial role in our study of \eqref{basic_pde2}. This is the space $\mps$ that was discussed in Section \ref{Sec1}.
 
\begin{definition} Let $\Om\subset\RR^n$ be an open set. We define $\mps$ to be the set of all functions $f\in L^p(\Om)$ such that there exists $C>0$ such that
\begin{equation}\label{cap}
\int_{K}|f|^p dx \leq C \cpt(K)
\end{equation}
for all compact  sets $K\subset\Om$  The norm of $f\in\mps$ is given by
 $$\norm{f}_{\mps} := \sup_{K\subset\Om} \left[ \frac{\int_{K}|f|^p dx}{\cpt(K)}\right]^{\frac{1}{p}}, $$
where the  sets $K$ vary over compact sets of $\Om$ such that $\cpt(K)>0$. 
\end{definition}

\begin{remark} For $f\in\mps$, we will always implicitly   extend $f$ by zero to $\RR^n\setminus\Om$, then inequality \eqref{cap} actually holds for all compact sets $K\subset\RR^n$. Thus by Theorem \ref{CHAR}  we see that    $f\in\mps$ if and only if there exists $C>0$ such that 
\begin{equation}\label{MA}
\int_{\Om} |\varphi|^p |f|^pdx \leq C \int_{\RR^n} (|\nabla\varphi|^p + |\varphi|^p) dx 
\end{equation}
for all $\varphi\in C_c^\infty(\RR^n)$. That is,
$f\in\mps$ if and only if $f\chi_{\Om}\in M(W^{1,p}(\RR^n)\rightarrow L^p(\RR^n))$. Moreover, the best constants $C$ in  \eqref{cap} and \eqref{MA}
are equivalent, i.e., their ratio is bounded from above and below by positive constants independent of $f$. 
In particular, inequalities \eqref{nablau-cond-Rn} and \eqref{Vec-cond} simply mean that $|F|^{\frac{1}{p-1}}$
and $|\nabla u|$ belong to ${\rm M}^{1,p}(\Om)$; and the best constants in  \eqref{nablau-cond-Rn} and \eqref{Vec-cond} are equivalent to 
$\norm{|\nabla u|}_{{\rm M}^{1,p}(\Om)}^{p}$ and $\norm{|F|^{\frac{1}{p-1}}}_{{\rm M}^{1,p}(\Om)}^{p}$, respectively.
\end{remark}

The following result will be useful to us.
\begin{lemma} \label{mutoF0} Suppose that $\mu$ is a finite sign measure in $\Om$ such that 
\begin{equation}\label{mucon}
|\mu|(K) \leq C_1 \, {\rm Cap}_{1,p}(K)
\end{equation}
holds for all compact sets $K\subset\Om$. Then we can write $\mu=\dv F$ in $\mathcal{D}'(\Om)$ for a vector field $F$ such that 
$|F|^{\frac{1}{p-1}}\in\mps$.  Moreover,
\begin{equation}\label{mutoF}
\norm{|F|^{\frac{1}{p-1}}}_{\mps} \leq C(n,p,{\rm diam}(\Om)) \, (C_1)^{\frac{1}{p-1}}. 
\end{equation}
\end{lemma}
\begin{proof}
After extending $\mu$ by zero outside $\Om$, we may write $\mu = {\rm div}\, F$ in the sense of distributions in $\Om$, where
\begin{equation*} 
F(x)=-\int_{B}\nabla_x G(x,y) d\mu(y).
\end{equation*}
Here $B$ is a ball of radius ${\rm diam}(\Om)$ containing $\Om$ and  $G(x,y)$ is the Green function with zero boundary
condition associated to $-\Delta$ on $B$. Note that we have 
\begin{equation*}
|\nabla_x G(x, y)|\leq C\, |x-y|^{1-n} \leq C(n, {\rm diam}(\Om))\, G_1(x-y)
\end{equation*} 
 for all $x, y\in B$ with $x\not=y$. Thus  $|F(x)| \leq C \, {\rm \bf G}_1(|\mu|)(x)$  which 
by Theorem \eqref{CHAR} yields that $|F|^{\frac{1}{p-1}}\in\mps$ along with estimate \eqref{mutoF}.
Here note that as $|\mu|$ is zero outside $\Om$, \eqref{mucon} actually holds for all compact sets $K\subset\RR^n$. 
\end{proof}

We now come to the key capacitary estimate that will  make it possible to obtain a solution of \eqref{basic_pde2} with strong regularity at the boundary of $\Om$. This important estimate was the main motivation of our earlier work \cite{AP1}.

\begin{lemma} \label{capboundu} Let $\mathcal{A}(x, \xi)$ and $\Om$ satisfy Assumptions 1 and 3. That is, we assume $\mathcal{A}$ satisfies \eqref{cond2}-\eqref{cond1}; $\mathcal{A}$ is $(\gamma, R)$-BMO; and
 $\Om$ is \gflt for a sufficiently small $\gamma= \gamma(n,p,\Lambda_0, \Lambda_1)>0$. Suppose that  $F$ is a vector field such that $|F|^{\frac{1}{p-1}}\in \mps$ and  $u\in \wpo$
is the unique solution of the equation
\begin{equation} \label{homoequa}
\dv \aa(x, \nabla u) = \dv F  \text{ in } \Omega, \qquad u = 0  \text{ on } \partial \Omega. 
\end{equation}
Then we have $|\nabla u|\in \mps$ with  
\begin{equation}\label{capbound-AP}
\norm{|\nabla u|}_{\mps} \leq C \norm{|F|^{\frac{1}{p-1}}}_{\mps},
\end{equation}
where $C=C(n,p,  \Lambda_0, \Lambda_1, {\rm diam}(\Om)/R_{0})$. Moreover, there exists a positive constant $C_0=C_0(n,p,{\rm diam}(\Om))$ such that 
 for any $\mu\in (0, \mu_0]$ with 
$$\mu_0=(\Lambda_0/C_0)^{\frac{1}{p-1}} \norm{|F|^{\frac{1}{p-1}}}_{\mps}^{-1},$$ we have $e^{\mu |u|}-1\in W^{1,p}_0(\Om)$ with
\begin{equation}\label{exp-W}
 \norm{e^{\mu |u|}-1}_{W^{1,p}_0(\Om)} \leq    C \mu \norm{F}_{L^{\frac{p}{p-1}}(\Om)}^{\frac{1}{p-1}},
 \end{equation}
where $C=C(p,\Lambda_0)$.

\end{lemma}

\begin{proof} Let  $u\in \wpo$ be the unique solution of \eqref{homoequa}. In \cite{AP1}, we showed that there exists $\gamma= \gamma(n,p,\Lambda_0, \Lambda_1)>0$
so that under Assumptions 1 and 3, the capacitary bound \eqref{capbound-AP} holds with a constant $C=C(n,p,\Lambda_0, \Lambda_1, {\rm diam}(\Om)/R_{0})$; see \cite[Corollary 1.8]{AP1}. We mention that the proof of  \eqref{capbound-AP} is based on an end-point weighted estimate obtained in \cite[Theorem 1.5]{AP1} and  a lemma of Verbitsky
\cite[Lemma 3.1]{MV}. In fact, the weighted estimate in \cite[Theorem 1.5]{AP1} was originally motivated from our  study of \eqref{basic_pde2}. This is also where we need the \gbmo condition on $\mathcal{A}$ and the $(\gamma,R_0)$-Reifenberg flatness condition on $\Om$.

To verify \eqref{exp-W}, let $T_s$, $s>0$, denote the two-sided truncation operator at level $s$, i.e., 
\begin{equation}\label{trun-op}
T_s(r)=r \ {\rm ~if~} |r|\leq s \quad {\rm and} \quad  T_s(r)= {\rm sign}(r) s \ {\rm ~if~} |r|>s.
\end{equation}

 For $s, \mu>0$ we define
$$u_s= T_s(u) \quad {\rm and} \quad w_s={\rm sign}(u) [e^{\mu |u_s|} -1]/\mu,$$ 
where ${\rm sign}(u)=0$ if $u=0$, ${\rm sign}(u)=1$ if $u>0$, and ${\rm sign}(u)=-1$ if $u<0$.

Note that  we have   $\nabla w_s = e^{\mu |u_s|} \nabla u_s= (e^{\mu |u_s|} \nabla u) \chi_{\{|u|\leq s\}}$ and thus if we let  
$$v_s=e^{ \delta |u_s|} w_s, \qquad \delta=(p-1)\mu,$$
then it holds that 
\begin{eqnarray}\label{nabvs}
\nabla v_s &=& \Big[e^{\de |u_s|} \nabla w_s + \de |w_s| e^{\de |u_s|}  \nabla u \Big]\chi_{\{|u|\leq s\}}\\
&=& \Big[e^{p\mu |u_s|} \nabla u + \de |w_s| e^{\de |u_s|}  \nabla u \Big]\chi_{\{|u|\leq s\}},\nonumber
\end{eqnarray}
since $\mu + \de = p\mu$.

Using $v_s$ as a test function in \eqref{homoequa} and employing \eqref{nabvs},  we get
\begin{eqnarray*}
 \lefteqn{\itl \aa(x,  \nabla u) \cdot \nabla u\, e^{p\mu |u_s|} \chi_{\{|u|\leq s\}} \ dx }\\
  &=&- \itl \de |w_s| e^{\de |u_s|} \aa(x, \nabla u)\cdot \nabla u \chi_{\{|u|\leq s\}} \ dx  + \itl F \cdot \nabla v_s  \ dx. 
\end{eqnarray*}

By \eqref{cond2}-\eqref{cond1} the first term on the right-hand side is nonpositive, and thus we get
\begin{equation}\label{I12}
 \itl \aa(x,  \nabla u) \cdot \nabla u\,  e^{ p\mu |u_s|} \chi_{\{|u|\leq s\}} \ dx \leq \itl F \cdot \nabla v_s  \ dx. 
\end{equation}

Since  $\nabla w_s = e^{\mu |u_s|} \nabla u_s$, using  conditions \eqref{cond2}-\eqref{cond1}, we see that
 \begin{eqnarray} \label{I1}
  \itl \aa(x,  \nabla u) \cdot \nabla u \, e^{ p\mu |u_s|} \chi_{\{|u|\leq s\}} \ dx & \geq& \Lambda_0 \itl  |\nabla u|^p e^{ p\mu |u_s|} \chi_{\{|u|\leq s\}} \ dx \\
& =& \Lambda_0 \itl |\nabla w_s|^p \ dx. \nonumber
 \end{eqnarray}

On the other hand, as  $v_s=e^{\delta |u_s|} w_s=(1+\mu |w_s|)^{p-1} w_s$ we  have
\begin{eqnarray*} 
 \itl F \cdot \nabla v_s  \ dx &=& \itl F \cdot \nabla[(1+\mu |w_s|)^{p-1} w_s] \ dx \\
&=&  \itl F \cdot [(p-1)(1+\mu |w_s|)^{p-2}\nabla w_s \, {\rm sign}(w_s)\, \mu  w_s] dx+ \nonumber\\
&& +\, \itl F \cdot [(1+\mu |w_s|)^{p-1} \nabla w_s] dx \nonumber\\
&\leq & p\itl |F|  (1+\mu |w_s|)^{p-1} |\nabla w_s| dx.\nonumber
\end{eqnarray*}

Using the inequality
$$(1+\mu |w_s|)^{p-1}\leq 2\mu^{p-1}|w_s|^{p-1} + C(p),$$ 
 and H\"older's inequality in the above bound  we then have

\begin{eqnarray*}
\itl F \cdot \nabla v_s  \ dx &\leq& 2  \mu^{p-1} \, p \itl |F|   |w_s|^{p-1} |\nabla w_s|dx +  C(p) \norm{F}_{L^{\frac{p}{p-1}}(\Om)} \norm{\nabla w_s}_{L^{p}(\Om)}\\
&\leq& 2  \mu^{p-1} \, p \left(\itl |F|^{\frac{p}{p-1}}   |w_s|^{p} dx\right)^{\frac{p-1}{p}} \norm{\nabla w_s}_{L^{p}(\Om)}\\
&&  +\,   C(p) \norm{F}_{L^{\frac{p}{p-1}}(\Om)} \norm{\nabla w_s}_{L^{p}(\Om)}.
\end{eqnarray*}

Note that by assumption $|F|^{\frac{1}{p-1}}\in \mps$, Theorem \ref{CHAR}, and Poincar\'e's inequality we have 
\begin{eqnarray*}
\left(\itl |F|^{\frac{p}{p-1}}   |w_s|^{p} dx\right)^{\frac{p-1}{p}} &\leq& C \norm{|F|^{\frac{1}{p-1}}}_{\mps}^{p-1} \left[\int_{\Om}(|\nabla w_s|^p +|w_s|^p)dx\right]  ^{\frac{p-1}{p}}\\
&\leq&  C(n,p, {\rm diam}(\Om)) \norm{|F|^{\frac{1}{p-1}}}_{\mps}^{p-1} \Big(\int_{\Om}|\nabla w_s|^pdx \Big)^{\frac{p-1}{p}}.
\end{eqnarray*}

 This gives
\begin{eqnarray}\label{I2}
\itl F \cdot \nabla v_s  \ dx &\leq&  C_1  \mu^{p-1}  \norm{|F|^{\frac{1}{p-1}}}_{\mps}^{p-1} \norm{\nabla w_s}_{L^{p}(\Om)}^{p}\\
&&  +\,    C(p) \norm{F}_{L^{\frac{p}{p-1}}(\Om)} \norm{\nabla w_s}_{L^{p}(\Om)},\nonumber
\end{eqnarray}
where $C_1=C_1(n,p, {\rm diam}(\Om))$. At this point we combine estimates \eqref{I1} and \eqref{I2} in equality \eqref{I12} to obtain the following bound
\begin{equation*}
\left(\Lambda_0-C_1  \mu^{p-1}  \norm{|F|^{\frac{1}{p-1}}}_{\mps}^{p-1} \right) \norm{\nabla w_s}^{p-1}_{L^{p}(\Om)}\leq   C(p) \norm{F}_{L^{\frac{p}{p-1}}(\Om)}.
\end{equation*}

This gives 
\begin{equation*}
\norm{\nabla w_s}_{L^{p}(\Om)}\leq     C(p) (\Lambda_0/2)^{\frac{-1}{p-1}} \norm{F}_{L^{\frac{p}{p-1}}(\Om)}^{\frac{1}{p-1}},
\end{equation*}
provided $$\mu\leq\left(\frac{\Lambda_0}{2C_1}\right)^{\frac{1}{p-1}} \norm{|F|^{\frac{1}{p-1}}}_{\mps}^{-1}.$$

Finally,  letting $s\nearrow +\infty$ we obtain the desiblack estimate in $W^{1,p}_0(\Om)$ for $e^{\mu|u|}-1$.
\end{proof}

The following convergence result, shown in \cite{BM},  will be  important for us in the proof of Theorem \ref{main2}(ii). 
 \begin{theorem}[\cite{BM}]
\label{boccardo_murat}
 Suppose that $\mathcal{A}$ satisfies \eqref{cond2}-\eqref{cond1}. For each $k>0$, let $w_k\in W^{1,p}(\Om)$ be a solution to the equation 
$$-\dv \aa(x,\nabla w_k) =   m_k + h_k  \quad \text{~in~} \mathcal{D}'(\Om),$$
and assume  that 
\begin{itemize}
 \item $w_k \rightarrow w$ weakly in $W^{1,p}(\Om)$, strongly in $L^p_{\rm loc}(\Om)$, and a.e. in $\Om$;
\item $h_k \rightarrow h$ in $(W_{0}^{1,p}(\Om))^{*}$;
\item $m_k$ is bounded in the space of finite Radon measures in $\Om$. That is, 
$|\langle m_k, \phi \rangle| \leq C_K \|\phi\|_{L^{\infty}(\Omega)}$ for all $\phi \in C_c^{\infty}(\Omega)$ with $\mathrm{spt}(\phi) \subset K$, where $C_K$ depends on $K$ but not on $k$.  
\end{itemize}

Then it holds that  $\nabla w_k \rightarrow \nabla w$ in $L^q(\Omega)$ for all $q<p$, and thus up to a subsequence  $\nabla w_k \rightarrow \nabla w$ a.e. in $\Om$.
\end{theorem}

{\color{black}We will also need the following strong convergence result first proved in F. E. Browder \cite{Bro} (see also \cite[Lemma 5]{BMP}).
\begin{lemma}
\label{brow_lemma}
Under  \eqref{cond2}-\eqref{cond1}, assume that the following two hypotheses are satisfied:
\begin{gather*}
u_{\ep} \rightarrow u \quad \text{in} \ W_0^{1,p}(\Om) \text{ weakly  and  a.e.  in } \Om,\\
\int_{\Om} [ \aa(x,  \nabla u_{\ep}) - \aa(x,  \nabla u)] \cdot \nabla (u_{\ep}-u)\ dx  \rightarrow 0.
\end{gather*}
Then it holds that 
\[
u_{\ep} \rightarrow u \quad \text{ in }\ W_0^{1,p}(\Om) \ \text{ strongly}.
\]

\end{lemma}
}

\section{Equations with general structures and main results}\label{main-AB}

 We are now ready to state the main result of the paper regarding the existence theory  for equation \eqref{basic_pde2}
in the space $\mps$. From our discussion on  $\mps$, we see that Theorem \ref{weakzero} is just a special case of  the following more general result.    

\begin{theorem} \label{main2} {\rm (i)} Suppose that $\aa(x,\xi)$ satisfies the first inequality in \eqref{cond1}, and 
$\mathcal{B}(x, s, \xi)$ satisfies the first inequality \eqref{Bcond}. If 
equation \eqref{basic_pde2} has a solution $u \in W^{1,p}_0(\Om)$ with $|\nabla u| \in \mps$, then there exists a vector field $F$  such that $\sigma = \dv F$ and $|F|^{\frac{1}{p-1}}\in \mps$ satisfying 
the estimate
\begin{equation*}
\norm{|F|^{\frac{1}{p-1}}}_{\mps} \leq C \left\{ \norm{\nabla u}_{\mps} + \norm{\nabla u}_{\mps}^{\frac{p}{p-1}}  + \norm{\nabla u}_{\mps}^{\frac{m}{p-1}}\right\}.
\end{equation*}
\noindent {\rm (ii)} Let $\mathcal{A}, \mathcal{B}$, and $\Om$ satisfy Assumptions 1, 2, and 3. That is, we assume $\mathcal{A}, \mathcal{B}$ satisfy \eqref{cond2}-\eqref{Bcond}; $\mathcal{A}$ is $(\gamma, R)$-BMO; and
 $\Om$ is \gflt for a sufficiently small $\gamma= \gamma(n,p,\Lambda_0, \Lambda_1)>0$.
Suppose that $\sigma= \dv F$ for a vector field $F$ such that $|F|^{\frac{1}{p-1}}\in \mps$. There is a positive number $\lambda_0=\lambda_0(n,p,m, \Lambda_0, \Lambda_1, b_0, b_1, b_2, {\rm diam}(\Om), R_0)>0$ such that    if 
$$\norm{|F|^{\frac{1}{p-1}}}_{\mps}
\leq \lambda_0^{1/p},$$ 
then  equation \eqref{basic_pde2} has a solution  $u \in W^{1,p}_0(\Om)$ with $|\nabla u| \in \mps$ and 
$$\norm{|\nabla u|}_{{\rm M}^{1,p}(\Om)} \leq T_0  \norm{|F|^{\frac{1}{p-1}}}_{{\rm M}^{1,p}(\Om)}.$$
Moreover, $u$ satisfies $e^{\mu |u|}-1\in W^{1,p}_0(\Om)$ for all $\mu \in [0, \mu_0]$ where 
$\mu_0=C \norm{|F|^{\frac{1}{p-1}}}_{{\rm M}^{1,p}(\Om)}^{-1}$, and
\begin{equation}\label{emu-u}
 \norm{e^{\mu |u|}-1}_{W^{1,p}_0(\Om)} \leq    C(p, \Lambda_0)\, \mu \norm{F}_{L^{\frac{p}{p-1}}(\Om)}^{\frac{1}{p-1}}.
 \end{equation}
\end{theorem}

We now devote to the proof of Theorem \ref{main2}. We first start with  part (i):

\noindent {\bf Proof of  Theorem \ref{main2}(i)}. Let $u\in \wpo$ be  such that $|\nabla u|\in \mps$. Then for $F=|\nabla u|^{p-2}\nabla u$,  we have
that $u$ solves
\begin{equation*}
\Delta_p u = \dv F  \text{ in } \Omega, \qquad u = 0  \text{ on } \partial \Omega. 
\end{equation*}
 
Thus by Lemma  \ref{capboundu} we have $e^{\mu_0 |u|}-1\in W^{1,p}_0(\Om)$ with
\begin{equation*}
 \norm{e^{\mu_0 |u|}-1}_{W^{1,p}_0(\Om)} \leq    C(p) \mu_0 \norm{F}_{L^{\frac{p}{p-1}}(\Om)}^{\frac{1}{p-1}},
 \end{equation*}
where 
\begin{equation}\label{muzero}
\mu_0=c(n,p,{\rm diam}(\Om)) \norm{|F|^{\frac{1}{p-1}}}_{\mps}^{-1}=c(n,p,{\rm diam}(\Om)) \norm{|\nabla u|}_{\mps}^{-1}.
\end{equation}

Note that 
$$\norm{F}_{L^{\frac{p}{p-1}}(\Om)}^{\frac{1}{p-1}}=\norm{\nabla u}_{L^p(\Om)}\leq C({n,p,\rm diam}(\Om)) \norm{|\nabla u|}_{\mps},$$
and thus  we get
\begin{equation}\label{gradnorm}
 \norm{e^{\mu_0 |u|}-1}_{W^{1,p}_0(\Om)} \leq    C(n,p,{\rm diam}(\Om)).
 \end{equation}

On the other hand, for any $m_0$ such that $m m_0\geq 1$, we have 
$$|\mu_0 u|^{m m_0}\leq \ceil{mm_0}! (e^{\mu_0 |u|}-1),$$
where $\ceil{x}$ denotes the smallest integer larger than or equal to $x$. Hence, by Poincar\'e's inequality and \eqref{gradnorm} we find
$$\int_{\Om} |\mu_0 u|^{m m_0} dx\leq C \norm{e^{\mu_0 |u|}-1}_{W^{1,p}_0(\Om)}\leq C.$$

That is, we have 
\begin{equation}\label{norm-u}
\left(\int_{\Om} |u|^{m m_0} dx\right)^{\frac{1}{m m_0}}\leq  C \mu_0^{-1}\leq C \norm{|\nabla u|}_{\mps}
\end{equation}
by \eqref{muzero}. Moreover, by H\"older's inequality we see that   \eqref{norm-u} in fact holds for all $m_0>0$ with a constant $C=C(n,p,m,m_0, {\rm diam}(\Om))$. 
Thus for $m_0>1$, using H\"older's inequality we get
\begin{equation}\label{Lebesgue-meas}
\int_{K}|u|^m dx \leq  \norm{|u|^m}_{L^{m_0}(\Om)} |K|^{1-\frac{1}{m_0}}\leq C \,\norm{|\nabla u|}_{\mps}^{m} \, |K|^{1-\frac{1}{m_0}}
\end{equation}
for any compact set $K\subset \Om$.

We next define $\kappa=n/(n-p)$ if $1<p<n$ and $\kappa=2$ if $p\geq n$. Then by Sobolev's inequality, for any compact set $K$ we have 
$$|K|^{\frac{1}{\kappa}}\leq \left(\int_{\RR^n} \varphi^{\kappa p} dx\right)^{\frac{1}{\kappa}} \leq C \int_{\RR^n} (|\nabla \varphi|^p +\varphi^p) dx$$
for all $\varphi\in C_c^\infty(\RR^n), 0\leq \varphi\leq 1$ and $\varphi= 1$ on $K$. This gives $|K|^{\frac{1}{\kappa}}\leq C {\rm Cap}_{1,p}(K)$ and thus choosing $m_0$ in \eqref{Lebesgue-meas} so that $1-1/m_{0}=1/\kappa$ we get 
\begin{equation}\label{Lebesgue-meas-cap}
\int_{K}|u|^m dx \leq   C \,\norm{|\nabla u|}_{\mps}^{m} \, {\rm Cap}_{1,p}(K).
\end{equation}


 Now using \eqref{Lebesgue-meas-cap}, the first bound in \eqref{Bcond}, and the fact that $|\nabla u|\in \mps$ we have
\begin{equation}\label{Bux}
\int_{K}|\mathcal{B}(x, u, \nabla u)| dx \leq  \left\{b_0\norm{|\nabla u|}_{\mps}^{p}  + C b_1 \,\norm{|\nabla u|}_{\mps}^{m} \right\} {\rm Cap}_{1,p}(K)
\end{equation}
for all compact sets $K\subset \Om$. Thus by Lemma \ref{mutoF0} we can write $\mathcal{B}(x, u, \nabla u ) = \dv F_1$ 
for a vector field $F_1$ such that 
$$\norm{|F_1|^{\frac{1}{p-1}}}_{\mps}\leq   C\, \norm{|\nabla u|}_{\mps}^{\frac{p}{p-1}} + C\, \norm{|\nabla u|}_{\mps}^{\frac{m}{p-1}}.$$

Now assume in addition that $u$ is a solution of \eqref{basic_pde2}, then we have 
$$\sigma=-\dv \mathcal{A}(x,\nabla u)-\mathcal{B}(x, u, \nabla u) =\dv [-\mathcal{A}(x,\nabla u)-F_1].$$
Thus letting $F=-\mathcal{A}(x,\nabla u)-F_1$ and using the first bound in \eqref{cond1}, we get the desiblack result.

We next prove  part (ii) of Theorem \ref{main2}.

\noindent{\bf Proof of Theorem \ref{main2}(ii).}  
The proof of this part will be carried out in several steps. First, we approximate \eqref{basic_pde2} and then obtain existence and regularity for the approximate equation. Eventually, we will use the regularity and appropriate test functions to pass to the limit. 

We begin by setting, for each $T>0$,   
\begin{equation*}
 \ee_T := \left\{ \phi \in W_0^{1,1}(\Om):  \phi\in  \wpo \text{~and~} \norm{ |\nabla \phi|}_{\mps} \leq T \norm{|F|^{\frac{1}{p-1}}}_{\mps}\right\}.
\end{equation*}

 We shall impose the subset topology from $W_0^{1,1}(\Om)$ on the set $\ee_T$. In fact, we could also use in $E_T$ the strong topology of 
$W_0^{1,q}(\Om)$ for any $1<q<p$. However, there is a  problem with compactness that prevents us from using the natural topology of $W_0^{1,p}(\Om)$ for $E_T$. 

It is easy to see from the definition of $\ee_T$ and Fatou's lemma that $E_T$ is convex and closed under the strong topology of $W_0^{1,1}(\Om)$.

For $k>0$, we now define a function $\mathcal{H}_k(x,s,\xi)$ by letting  
$$\mathcal{H}_k(x,s,\xi):= \frac{\mathcal{B}(x,s,\xi)}{1+ |\mathcal{B}(x,s,\xi)| k^{-1}}.$$
Then $\mathcal{H}_k(x,s,\xi)$ also satisfies \eqref{Bcond} and $|\mathcal{H}_k(x,s,\xi)|\leq \min\{k, |\mathcal{B}(x,s,\xi)|\}$. For each $v\in \ee_T$ and each vector field $F$ such that $|F|^{\frac{1}{p-1}}\in \mps$ we let $S(v)$ denote the unique solution $u\in \wpo$ to the equation 
\begin{equation} \label{approx_two}
-\dv \aa(x,\nabla u) = \mathcal{H}_k(x, v, \nabla v) + \dv \, F  \text{ in } \Omega, \qquad u = 0  \text{ on } \partial \Omega. 
\end{equation}

The map $S: \ee_T\rightarrow W_0^{1,p}(\Om)$ is well defined as the right-hand side of the first equation in \eqref{approx_two} belongs to $(\wpo)^*$. Note that we have not yet made any choice on $T$.
We shall break the proof into three steps.

\noindent {\bf Step 1.} In this step we show that there exists $T_0>0$ independent of $k$ such that whenever 
\begin{equation}\label{F-small}
\norm{|F|^{\frac{1}{p-1}}}_{\mps}\leq 2 \min\left\{ (2T_0)^{-p}, (2T_0)^{\frac{-m}{m-p+1}}\right\}, 
\end{equation}  we have
\begin{equation} \label{Amap}
S: \ee_{T_0} \rightarrow  \ee_{T_0}.
\end{equation}

Indeed, since $|\mathcal{H}_{k}(x, v, \nabla v)|\leq |\mathcal{B}(x, v, \nabla v)|$, by calculations as in the proof of part (i) (see \eqref{Bux}) we have 
$$\int_{K}|\mathcal{H}_k(x, v, \nabla v)| dx \leq  \left\{ b_0\norm{|\nabla v|}_{\mps}^{p} + C b_1\,\norm{|\nabla v|}_{\mps}^{m} \right\} {\rm Cap}_{1,p}(K)$$
for any compact set $K$. Thus by Lemma \ref{mutoF0} 
we can write 
$\mathcal{H}_k(x, v, \nabla v)=\dv F_k$ for a vector field $F_k$ such that 
$$\norm{|F_k|^{\frac{1}{p-1}}}_{\mps}\leq C \left\{\norm{|\nabla v|}_{\mps}^{\frac{p}{p-1}} + \norm{|\nabla v|}_{\mps}^{\frac{m}{p-1}} \right\}.$$

Then from Lemma \ref{capboundu} we find that $|\nabla S(v)|\in \mps$ with 
\begin{eqnarray*}
\norm{|\nabla S(v)|}_{\mps} &\leq& C \, \norm{|F_k + F|^{\frac{1}{p-1}}}_{\mps}\\
&\leq& \overline{C} \left\{\norm{|\nabla v|}_{\mps}^{\frac{p}{p-1}} + \norm{|\nabla v|}_{\mps}^{\frac{m}{p-1}} +  \norm{|F|^{\frac{1}{p-1}}}_{\mps}\right\}
\end{eqnarray*}
for a constant $\overline{C}$ independent of $k$. 

We now choose $T_0=2 \overline{C}$. Then for $v\in \ee_{T_0}$ and $F$ satisfying \eqref{F-small} we have 
\begin{eqnarray*}
\norm{|\nabla S(v)|}_{\mps} &\leq&  \overline{C} \left(T_0^{\frac{p}{p-1}}\norm{|F|^{\frac{1}{p-1}}}_{\mps}^{\frac{1}{p-1}} + T_0^{\frac{m}{p-1}} \norm{|F|^{\frac{1}{p-1}}}_{\mps}^{\frac{m-p+1}{p-1}}\right) \norm{|F|^{\frac{1}{p-1}}}_{\mps}\\
&& +\, \overline{C} \norm{|F|^{\frac{1}{p-1}}}_{\mps}\\
&\leq& 2 \overline{C} \norm{|F|^{\frac{1}{p-1}}}_{\mps}= T_0 \norm{|F|^{\frac{1}{p-1}}}_{\mps}.
\end{eqnarray*}

This gives   $S(v)\in \ee_{T_0}$ and thus \eqref{Amap} follows.

\noindent {\bf Step 2.} We now prove that for each $k>0$ there exists a  solution $u_k\in \ee_{T_0}$ to the approximate equation 
\begin{equation}
 \label{approx_one}
-\dv \aa(x,\nabla u_k) = \mathcal{H}_k(x, u_k, \nabla u_k) + \dv \, F  \text{ in } \Omega, \qquad u_k = 0  \text{ on } \partial \Omega. 
\end{equation}

To that end,  we shall use Schauder's Fixed Point Theorem to obtain a fixed point for the map
$S: \ee_{T_0} \rightarrow \ee_{T_0}$.  
 Since we already know that $\ee_{T_0}$ is closed and convex, it remains to show that 
 $S: \ee_{T_0} \rightarrow \ee_{T_0}$ is continuous and $S(\ee_{T_0})$ is pre-compact (under the strong topology of $W_0^{1,1}(\Om)$).

To prove continuity,  let $\{v_l \}\subset\ee_{T_0}$ be a sequence such that 
$v_l \rightarrow v$ strongly in $W_0^{1,1}(\Om)$. This combined with the fact that $\{S(v_l)\}$ is bounded in $\wpo$, there is a subsequence, also denoted by $\{v_l\}$ for simplicity, and a function $u\in \wpo$ such that:
    
\noindent (a) $S(v_l) \rightarrow u$ weakly in $\V$, strongly in $L^{p}(\Om)$, and a.e. in $\Om$,

\noindent (b) $\{\mathcal{H}_k(x, v_l, \nabla v_l)\}_{l}$ is uniformly bounded in the space of finite
signed measures in $\Om$ for each fixed $k>0$.

Recall that we have 
\begin{equation}\label{Avtou}
-\dv \aa(x,\nabla [S(v_l)]) = \mathcal{H}_k(x, v_l, \nabla v_l) + \dv \, F \quad \text{~in~} \mathcal{D}'(\Omega). 
\end{equation}
Thus by Theorem \ref{boccardo_murat}  we have $S(v_l)\rightarrow u$ in $W^{1,q}_0(\Om)$ for any $1\leq q<p$, and up to another subsequence we have  $\nabla [S(v_l)]\rightarrow \nabla u$ a.e. in $\Om$.

By \eqref{cond1} and Vitali's Convergence Theorem  we  have 
$$\mathcal{A}(x, \nabla[S(v_l)]) \rightarrow \mathcal{A}(x, \nabla u) \quad \text{strongly ~in~} L^{1}(\Om,\RR^n) {\rm ~and ~ weakly~ in~} L^{p/(p-1)}(\Om,\RR^n).$$

Up to another subsequence, it holds that $v_l \rightarrow v$ and $\nabla v_l \rightarrow \nabla v$ a.e. in $\Om$. Thus by  Dominated Convergence Theorem we have 
$$\mathcal{H}_k(x, v_l, \nabla v_l)\rightarrow
\mathcal{H}_k(x, v, \nabla v) \quad \text{~in~} L^{1}(\Om). $$
This is where we use the property that $|\mathcal{H}_k(x, v_l, \nabla v_l)| \leq k$ for all $l$.

Thus we can pass to the limit in equation \eqref{Avtou} to obtain that $u=S(v)$. So far we have found a subsequence of $\{v_{l_{j}}\}$ of $\{v_l\}$ such that $S(v_{l_j})\rightarrow S(v)$ in $W^{1,1}_0(\Om)$.   As the limit is independent of the subsequence it actually holds that the whole sequence  $S(v_{l})\rightarrow S(v)$ in $W^{1,1}_0(\Om)$.  This shows that the map $S:\ee_{T_0} \rightarrow \ee_{T_0}$ is continuous.

To prove pre-compactness, let $\{u_l\}=\{S(v_l)\}$ be a sequence in $S(\ee_{T_0})$, where $v_l\in \ee_{T_0}$. Then as above  there is a subsequence of $\{v_l\}$, also denoted by $\{v_l\}$, and a function $u\in \wpo$ such that properties (a) and (b) above hold.  Thus by Theorem \ref{boccardo_murat} again we have $S(v_l)\rightarrow u$ in $W^{1,1}_0(\Om)$. This shows that the set $S(\ee_{T_0})$ is pre-compact.

\noindent {\bf Step 3.} In this step we further restrict that 
\begin{equation*}
\norm{|F|^{\frac{1}{p-1}}}_{\mps} \leq \lambda_0^{\frac{1}{p}}, 
\end{equation*} 
 where 
$$\lambda_0^{\frac{1}{p}}:= \min\left\{2 \min\left\{ (2T_0)^{-p}, (2T_0)^{\frac{-m}{m-p+1}}\right\}, (b_2/\Lambda_0)^{-1}(\Lambda_0/C_0)^{\frac{1}{p-1}}\right\},$$
and $C_0=C_0(n,p,{\rm diam}(\Om))$ is as in Lemma \ref{capboundu}. Then we have
\begin{equation}\label{muzeroupbound} 
\mu_0=(\Lambda_0/C_0)^{\frac{1}{p-1}} \norm{|F|^{\frac{1}{p-1}}}_{{\rm M}^{1,p}(\Om)}^{-1}\geq b_2/\Lambda_0.
\end{equation} 

Let $u_k$ be as in Step 2. By Lemma \ref{capboundu}, we have $e^{\mu |u_k|}-1\in W^{1,p}_0(\Om)$ for $\mu \in [0, \mu_0]$, with
\begin{equation}\label{exp-W-1}
 \norm{e^{\mu |u_k|}-1}_{W^{1,p}_0(\Om)} \leq    C \mu \norm{F}_{L^{\frac{p}{p-1}}(\Om)}^{\frac{1}{p-1}},
 \end{equation}
where $C=C(p,\Lambda_0)$. Then by  Rellich's compactness theorem, there is a subsequence, still denoted by $\{u_k\}$, such that  
$$u_{k} \xrightharpoonup{k} u \text{~weakly~in~} \wpo,    \text{~strongly~in~} L^p(\Om), \text{~and~a.e.~in~} \Om,$$
 for a function $ u \in \wpo$ such that \eqref{emu-u} holds for all $\mu \in [0, \mu_0]$.

We now claim that 
\begin{equation}\label{strongconv}
u_k\rightarrow u \quad \text{strongly~in~} \wpo \text{~as~} k\nearrow\infty,
\end{equation}
and thus  we can pass to the limit in \eqref{approx_one} to verify that $u$ is  a solution to \eqref{basic_pde2}.

To prove \eqref{strongconv}, we  write
$$\nabla u_k-\nabla u= \nabla T_s(u_k)- \nabla T_s(u) + \nabla G_s(u_k)- \nabla G_s(u),$$
where $s>0$ and $G_s(r):=r-T_s(r)$, $r\in\RR$, and $T_s$ is as defined in \eqref{trun-op}. 
Thus for every $s>0$ we have 
\begin{eqnarray}\label{tailplushead}
\norm{\nabla u_k-\nabla u}_{L^p(\Om)} &\leq& \sup_{k>0}\norm{\nabla G_s(u_k)- \nabla G_s(u)}_{L^p(\Om)}\\
&& + \, \norm{\nabla T_s(u_k)- \nabla T_s(u)}_{L^p(\Om)}.\nonumber
\end{eqnarray}

Note that for $\mu\in (0, \mu_0]$ we find 
 \begin{eqnarray*}
  \itl |\nabla G_s(u_k)|^p \ dx&=&\itl[\{|u_k|>s\}] |\nabla u_k|^p \ dx \\
	&=&  \frac{1}{\mu}\itl[\{|u_k|>s\}] e^{-p\mu|u_k|}|\nabla (e^{\mu |u_k|}-1)|^p \ dx \\
  &\leq& \frac{1}{\mu} e^{-p\mu s} \norm{e^{\mu |u_k|}-1}_{W^{1,p}_0(\Om)}^p.
 \end{eqnarray*}

By \eqref{exp-W-1}, this yields that  
\begin{equation}\label{tail}
\lim_{s\rightarrow\infty}\ \sup_{k>0}\norm{\nabla G_s(u_k)- \nabla G_s(u)}_{L^p(\Om)} =0.
\end{equation}

On the other  hand, by \eqref{muzeroupbound} and Lemma \ref{truncate-conv} below it holds that 

\begin{equation}\label{head}
\lim_{k\rightarrow\infty}\norm{\nabla T_s(u_k)- \nabla T_s(u)}_{L^p(\Om)} =0 \quad \text{for each } s>0.
\end{equation}

Thus combining \eqref{tailplushead}, \eqref{tail}, and \eqref{head} we obtain  convergence \eqref{strongconv} as desiblack. 
This completes the proof of the theorem.

We are now left with the proof of the following lemma.
\begin{lemma} \label{truncate-conv}Let $\mathcal{A}, \mathcal{B}$ satisfy \eqref{cond2}-\eqref{Bcond} and let $F$ be a vector field in $L^{\frac{p}{p-1}}(\Om,\RR^n)$.  For each $k>0$, 
let $u_k\in W^{1,p}_0(\Om)$ be a solution of \eqref{approx_one}. Suppose that there exists $\mu_0 \geq b_2/\Lambda_0$ such that $\{e^{\mu_0|u_k|}-1\}$ is uniformly bounded in  $W^{1,p}_0(\Om)$. Suppose also that  
$$u_{k} \xrightharpoonup{k} u \text{~weakly~in~} \wpo,    \text{~strongly~in~} L^p(\Om), \text{~and~a.e.~in~} \Om,$$
 for a function $ u \in \wpo$.
Then we have the convergence   \eqref{head} for any $s>0$.
\end{lemma}

\begin{proof} Let $s>0$ be fixed. For any $j\geq s$ we define   $$ v_k = \exptk \psi(z_k), 
$$ where $z_k = T_s(u_k) - T_s(u)$ and $\psi$ is a $C^1$ and increasing function from $\RR$ to $\RR$ satisfying 
\begin{equation}\label{psicond}
\psi(0) = 0 \quad \text{and} \quad \psi' - \frac{b_0+  \Lambda_1 \mu_0}{\Lambda_0} |\psi| \geq 1.
\end{equation}

As in \cite{FM2} (see also  \cite{BBM, FM3}), using $v_k$ test function in \eqref{approx_one}, we  have 
\begin{eqnarray*}
  \lefteqn{\itl \axgrad{\nabla u_k} \cdot \exptk \psi'(z_k) \nabla z_k \dx}\\
	&=& \itl \Big[\mathcal{H}_k(x, u_k, \nabla u_k)  - \mu_0 \mathcal{A}(x,\nabla u_k) \cdot \nabla T_j(u_k) {\rm sign}(u_k)\Big]  \exptk \psi(z_k) \dx  \\
&& + \ \int_{\Om} F\cdot \nabla[ \exptk \psi(z_k)] dx.  
\end{eqnarray*}

Note that the term on the left-hand side  in the above equality can be written as
\begin{eqnarray*}
  \lefteqn{\itl \axgrad{\nabla u_k} \cdot  (\nabla T_s(u_k) - \nabla T_s(u)) \exptk \psi'(z_k)\dx}  \\
& =& \itl[\{\abs{u_k}\leq s\}] \{\axgrad{\nabla T_s(u_k)} - \axgrad{\nabla T_s (u)} \}\cdot \\
&& \qquad \qquad \qquad \cdot\,  (\nabla T_s(u_k) - \nabla T_s(u)) \exptk \psi'(z_k)\dx \\
& &\quad + \itl[\{\abs{u_k}\leq s\}]\axgrad{\nabla T_s (u)} \cdot  (\nabla T_s(u_k) - \nabla T_s(u)) \exptk \psi'(z_k)\dx \\
& &\quad + \itl[\{\abs{u_k}>s\}] \axgrad{\nabla u_k} \cdot  (- \nabla T_s(u)) \exptk \psi'(z_k)\dx.
\end{eqnarray*}

Thus combining the last two equalities we obtain
\begin{equation}\label{First5}
I_1 - I_4=-I_2-I_3  +I_5,
\end{equation}
where we define that 
\begin{eqnarray*}
I_1&=&\itl[\{\abs{u_k}\leq s\}] \{\axgrad{\nabla T_s(u_k)} - \axgrad{\nabla T_s (u)} \}\cdot\\
&& \qquad \qquad \qquad \cdot\, (\nabla T_s(u_k) - \nabla T_s(u)) \exptk \psi'(z_k)\dx,
\end{eqnarray*}
$$I_2=\itl[\{\abs{u_k}\leq s\}]\axgrad{\nabla T_s (u)} \cdot  (\nabla T_s(u_k) - \nabla T_s(u)) \exptk \psi'(z_k)\dx,$$
$$I_3=\itl[\{\abs{u_k}>s\}] \axgrad{\nabla u_k} \cdot  (- \nabla T_s(u)) \exptk \psi'(z_k)\dx,$$
$$I_4=\itl \left[\mathcal{H}_k(x, u_k, \nabla u_k)  - \mu_0 \mathcal{A}(x, \nabla u_k) \cdot \nabla T_j(u_k) {\rm sign}(u_k)\right]  \exptk \psi(z_k) \dx,$$
and 
$$I_5=\int_{\Om} F\cdot \nabla[ \exptk \psi(z_k)] dx.$$

We further have 
\begin{equation}\label{I2-I4}
I_1-I_4=I_1-I_4^1-I_4^2,
\end{equation}
where 
$$I_4^1:= \itl[\{\abs{u_k}> s\}] \{ \dots\}\, dx, \qquad \qquad I_4^2:=\itl[\{\abs{u_k}\leq s\}] \{ \dots\}\, dx,$$
with $\{\dots\}$ being the integrand in $I_4$.


Since   $\abs{\nabla T_j(u_k)}\leq  \abs{\nabla u_k}$,   using \eqref{cond2}-\eqref{Bcond}  we get
\begin{eqnarray*}
|I_4^2| &\leq&     \itl[\{\abs{u_k}\leq s\}] \Big(b_0 \abs{\nabla u_k}^p + b_1 |u_k|^m  + \mu_0 \Lambda_1 \abs{\nabla u_k}^{p} \Big)    \exptk\  \abs{\psi(z_k)} \   dx \\
&\leq& \frac{b_0 + \mu_0 \Lambda_1}{\Lambda_0}  \itl[\{\abs{u_k}\leq s\}] \mathcal{A}(x, \nabla T_s(u_k)) \cdot \nabla T_s(u_k)   \exptk\  \abs{\psi(z_k)} \   dx \\
&& +\  \itl[\{\abs{u_k}\leq s\}] b_1 s^m   e^{\mu_0 s}\  \abs{\psi(z_k)}\    \dx,
\end{eqnarray*}

 Thus, with $M=(b_0 + \mu_0 \Lambda_1)/\Lambda_0$,
we find that 
\begin{eqnarray*} 
\qquad |I_4^2|   & \leq & M \itl[\{\abs{u_k}\leq s\}]  \left[ \axgrad{\nabla T_s(u_k)} -\axgrad{\nabla T_s(u)} \right] \cdot \\
&& \qquad \qquad \qquad \cdot \left[ \nabla T_s(u_k) - \nabla T_s(u) \right]    \exptk\  \abs{\psi(z_k)}  dx \nonumber\\
&& + \  M\itl[\{\abs{u_k}\leq s\}]\axgrad{\nabla T_s(u)} \cdot \left[ \nabla T_s(u_k) - \nabla T_s(u) \right] \exptk\  \abs{\psi(z_k)}    dx \nonumber\\
&&  + \   M\itl[\{\abs{u_k}\leq s\}]\axgrad{\nabla T_s(u_k)} \cdot \nabla T_s(u)  \exptk\  \abs{\psi(z_k)}   dx\nonumber \\
&& +\   \itl[\{\abs{u_k}\leq s\}] b_1 s^m   e^{\mu_0 s}\  \abs{\psi(z_k)}   dx.\nonumber
\end{eqnarray*}

On the other hand, using  the inequalities $$\axgrad{\nabla u_k}\cdot  \nabla T_j(u_k) \geq \Lambda_0 \abs{\nabla u_k}^p \chi_{\{\abs{u_k}\leq j\}}, \quad \chi_{\{\abs{u_k}>s\}} {\rm sign}(u_k) \psi(z_k) \geq 0,$$ and  \eqref{Bcond}, we have
%
\begin{eqnarray*}
 I_4^1 &=& \itl[\{\abs{u_k}>s\}] \left[ {\rm sign}(u_k) \mathcal{H}_k(x,u_k, \nabla u_k) -\mu_0 \axgrad{\nabla u_k} \cdot \nabla T_j(u_k) \right] \times \\
&& \qquad \qquad \qquad \qquad \times \ {\rm sign}(u_k) 
\exptk \psi(z_k) \dx \\
&\leq& \itl[\{\abs{u_k}>s\}] \left[ b_2 \abs{\nabla u_k}^p -\mu_0 \Lambda_0\abs{\nabla u_k}^{p} \chi_{\{\abs{u_k}\leq j\}} \right] {\rm sign}(u_k) \exptk \psi(z_k) \dx.
\end{eqnarray*}

 Thus since $\mu_0 \Lambda_0\geq b_2$ and $j\geq s$,   we get
\begin{eqnarray*}
 I_4^1  & \leq & \itl[\{\abs{u_k}>j\}] b_2 \abs{\nabla u_k}^p  {\rm sign}(u_k) \exptk \psi(z_k) \dx\\
& \leq & b_2 \max_{r\in[-2s,2s]}|\psi(r)| \  e^{\mu_0 j}\itl[\{\abs{u_k}>j\}]  \abs{\nabla u_k}^p   \dx\nonumber\\
&\leq& C e^{\mu_0 j}\, \frac{1}{\mu_0} \, \itl[\{|u_k|>j\}] e^{-p\mu_0|u_k|}|\nabla (e^{\mu_0 |u_k|}-1)|^p \ dx\nonumber\\
& \leq & C \, \frac{1}{\mu_0}\,   e^{\mu_0 j}  e^{-p\mu_0 j} \norm{e^{\mu_0 |u_k|}-1}_{W^{1,p}_0(\Om)}^p.\nonumber
\end{eqnarray*}

As the sequence $\{e^{\mu_0 |u_k|}-1\}$ is uniformly bounded in  $W^{1,p}_0(\Om)$, this yields  that 
\begin{equation}\label{I-41}
\limsup_{j\rightarrow\infty} \ \sup_{k>0} I_4^1 \leq 0.
\end{equation}

Let $D_k$ be the nonnegative function 
$$D_k:=(\axgrad{\nabla T_s(u_k)} - \axgrad{\nabla T_s (u)} )\cdot  (\nabla T_s(u_k) - \nabla T_s(u)).$$
Then by  \eqref{psicond},
$$\itl[\{\abs{u_k}\leq s\}] D_k \ dx \leq  \itl[\{\abs{u_k}\leq s\}] D_k \, \exptk \, (\psi' - M |\psi|)\ dx.$$

Thus combining this with \eqref{I2-I4}-\eqref{I-41}, we get
\begin{eqnarray*}
\lefteqn{\itl[\{\abs{u_k}\leq s\}] D_k \ dx \leq I_1- I_4 \, + } \nonumber\\
&&   +\   M\itl[\{\abs{u_k}\leq s\}]\axgrad{\nabla T_s(u)} \cdot \left[ \nabla T_s(u_k) - \nabla T_s(u) \right] \exptk\  \abs{\psi(z_k)}    dx \nonumber\\
&&  + \   M\itl[\{\abs{u_k}\leq s\}]\axgrad{\nabla T_s(u_k)} \cdot \nabla T_s(u)  \exptk\  \abs{\psi(z_k)}   dx\nonumber \\
&& +\   \itl[\{\abs{u_k}\leq s\}] b_1 s^m   e^{\mu_0 s}\  \abs{\psi(z_k)}   dx + \varepsilon \nonumber
\end{eqnarray*}
for any $\varepsilon>0$ provided $j=j(\varepsilon)$ is sufficiently large. 

Our next goal is to apply $\limsup_{k\rightarrow\infty}$ to both sides of  the above inequality. To that end, note that $\psi(0)=0$, $z_k \xrightarrow{k} 0$ a.e., $\{\nabla T_s(u_k) - \nabla T_s(u)\}$ and $\{\axgrad{ \nabla T_s(u_k)}\}$ are uniformly bounded in $L^{p}(\Om,\RR^n)$ and in $L^{p/(p-1)}(\Om,\RR^n)$, respectively. Thus by H\"older's inequality and  Dominated Convergence Theorem we find

 
$$\lim_{k\rightarrow\infty} \itl[\{\abs{u_k}\leq s\}]\axgrad{\nabla T_s(u)} \cdot \left[ \nabla T_s(u_k) - \nabla T_s(u) \right] \exptk\  \abs{\psi(z_k)}    dx,$$ 
$$\lim_{k\rightarrow\infty}\itl[\{\abs{u_k}\leq s\}]\axgrad{\nabla T_s(u_k)} \cdot \nabla T_s(u)  \exptk\  \abs{\psi(z_k)}   dx=0,$$ 
and 
$$\lim_{k\rightarrow\infty}\itl[\{\abs{u_k}\leq s\}]  b_1 |s|^m    e^{\mu_0 s}  \abs{\psi(z_k)} dx=0. $$

Thus we get 
\begin{eqnarray}\label{I24low}
\limsup_{k\rightarrow\infty} \itl[\{\abs{u_k}\leq s\}] D_k\ dx&\leq& \limsup_{k\rightarrow\infty}(I_1 -I_4) +\varepsilon \\
&=& \limsup_{k\rightarrow\infty}(-I_2 -I_3+I_5) +\varepsilon\nonumber 
\end{eqnarray}
for any $\varepsilon>0$ provided $j=j(\varepsilon)$ is sufficiently large. Here we use \eqref{First5} in the last equality.

We  next claim that  for any  $j\geq s$ we have 
\begin{equation}\label{I235}
\lim_{k\rightarrow\infty}(-I_2 -I_3+I_5)=0.
\end{equation}
To prove this claim, we treat each term on the right-hand side separately as follows.

\noindent {\bf The term $I_2$:}  Since $u_k \xrightarrow{k} u$ a.e. and  $z_k \xrightarrow{k} 0$ a.e., it holds that  
 $$\axgrad{\nabla T_s (u)} \exptk \psi'(z_k) \xrightarrow{k} \axgrad{\nabla T_s (u)} e^{\mu_0\abs{T_j(u)}} \psi'(0) \quad {\rm a.e.}$$

Thus using the pointwise estimate,
 $$|\axgrad{\nabla T_s(u)} e^{\mu_0 |T_j(u_k)|} \psi'(z_k)| \leq  e^{\mu_0 j} \max_{r \in [-2s,2s]}\abs{\psi'(r)} 
 \Lambda_1|\nabla T_s(u)|^{p-1} $$
 and  Dominated Convergence Theorem, we have 
$$\axgrad{\nabla T_s (u)} \exptk \psi'(z_k) \xrightarrow{k} \axgrad{\nabla T_s (u)} e^{\mu_0\abs{T_j(u)}}\psi'(0)$$  
strongly in $L^{p/(p-1)}(\Om,\RR^n)$.  

Next, since $u_k$ is uniformly bounded in $\wpo$ and $T_s(u_k) \xrightarrow{k} T_s(u)$ a.e., we have $\nabla T_s(u_k) \xrightharpoonup{k} \nabla T_s(u)$ weakly in $L^p(\Om,\RR^n)$. 
On the other hand, since 
\begin{equation}\label{chiconv}
\chi_{\{\abs{u_k} \leq s\}} \xrightarrow{k} \chi_{\{\abs{u} \leq s\}}  {\rm~ a.e. ~in~} \Om\setminus\{\abs{u} = s\} {\rm ~and~} |\nabla T_s(u)| = 0 {\rm~ a.e.~ on~} \{\abs{u} = s\},
\end{equation}
 we  have  from  Dominated Convergence Theorem that 
\begin{equation*}
\nabla T_s(u)\chi_{\{\abs{u_k} \leq s\}} \xrightarrow{k} \nabla T_s(u)\chi_{\{\abs{u} \leq s\}}=\nabla T_s(u) \trm{strongly in} L^p(\Om,\RR^n). 
\end{equation*}

Thus, 
\begin{eqnarray}\label{weaknablaz}
\chi_{\{\abs{u_k} \leq s\}} (\nabla T_s(u_k) - \nabla T_s(u)) &=& \nabla T_s(u_k) - \nabla T_s(u)\chi_{\{\abs{u_k} \leq s\}}\\
 &\xrightharpoonup{k}& 0 \trm{weakly in} L^p(\Om,\RR^n).\nonumber
\end{eqnarray}

These convergences  imply  that $$\lim_{k\rightarrow\infty} I_2 =0.$$

\noindent {\bf The term $I_3$:} By \eqref{cond1},  $|\axgrad{\nabla u_k}|$ is uniformly bounded in $L^{p/(p-1)}(\Om)$. On the other hand, again by \eqref{chiconv} and  Dominated Convergence Theorem  we have $$|\chi_{\{\abs{u_k}>s\}}  (-\nabla T_s(u)) \exptk \psi'(z_k)| \xrightarrow{k} 0 \trm{strongly in} L^p(\Om).$$

Thus using H\"older's inequality we see that $$\lim_{k\rightarrow\infty}I_3=0.$$

\noindent {\bf The term $I_5$:} We have 
\begin{eqnarray}\label{I5term}
I_5&=&\mu_0 \itl F \cdot \exptk \psi(z_k) \nabla T_j(u_k) {\rm sign}(u_k)\ dx\\
&& +\itl F \cdot \exptk \psi'(z_k) \nabla z_k \ dx.\nonumber
\end{eqnarray}

As  $F  \exptk \psi(z_k)  \xrightarrow{k} (0, \dots,0)$ a.e. in  $\Om$, by  Dominated Convergence Theorem we find  
 $$ F \exptk \psi(z_k)   \xrightarrow{k} (0,\dots, 0) \trm{strongly in} L^{p/(p-1)}(\Om,\RR^n).$$
Since $\nabla T_j(u_k) {\rm sign}(u_k)$ is uniformly bounded in $L^p(\Om,\RR^n)$, we then conclude that 
\begin{equation}\label{de1}
\mu_0 \itl F\cdot \exptk \psi(z_k) \nabla T_j(u_k) {\rm sign}(u_k)\ dx \xrightarrow{k} 0.
\end{equation}

Again, by  Dominated Convergence Theorem we have
$$ F \exptk \psi'(z_k) \xrightarrow{k} F \expt[\abs{T_j(u)}] \psi'(0) \trm{strongly in} L^{p/(p-1)}(\Om,\RR^n).$$ 

Thus using \eqref{weaknablaz} and $\nabla z_k= \nabla T_s(u_k) - \nabla T_s(u)$, we obtain that  
$$\itl[\{\abs{u_k}\leq s\}] F \cdot \exptk \psi'(z_k) \nabla z_k \ dx \xrightarrow{k} 0.$$

On the other hand, 
\begin{eqnarray*}
\lefteqn{\itl[\{\abs{u_k}> s\}] F \cdot \exptk \psi'(z_k) \nabla z_k \  dx}\\
&=&\itl F\cdot \exptk \psi'(z_k) (-\nabla T_s(u)) \chi_{\{\abs{u_k}> s\}}\ dx\\
&\xrightarrow{k}& 0,
\end{eqnarray*}
by \eqref{chiconv}, H\"older's inequality,  and  Dominated Convergence Theorem.

Combining the last two limits, we obtain
\begin{equation}\label{de2}
\itl F\cdot \exptk \psi'(z_k) \nabla z_k \ dx \xrightarrow{k} 0.
\end{equation}

Hence combining \eqref{I5term}, \eqref{de1}, and \eqref{de2},   we conclude that 
$$\lim_{k\rightarrow\infty} I_5=0.$$ 

Thus we have shown that the limit \eqref{I235} holds. Then
in view of \eqref{I24low} and the fact that $D_k\geq 0$  we find that 
\begin{equation*}
\itl[\{\abs{u_k}\leq s\}] D_k \ dx \xrightarrow{k} 0.
\end{equation*}

On the other hand, by \eqref{chiconv}, 
\begin{eqnarray*}
\chi_{\{\abs{u_k}> s\}} D_k
&=&\chi_{\{\abs{u_k}> s\}}  [\axgrad{0} - \axgrad{\nabla T_s (u)}]\cdot  (-\nabla T_s(u))\\
&\xrightarrow{k}& 0 \quad \text{a.e.,}
\end{eqnarray*}
which implies that 
$\itl[\{\abs{u_k} > s\}] D_k \  dx \xrightarrow{k} 0$. Thus we obtain
\begin{equation}\label{Brow}
\itl[\Om] D_k \  dx \xrightarrow{k} 0.
\end{equation}

Finally, with \eqref{Brow} we can apply  {\color{black} Lemma \ref{brow_lemma} to} conclude the proof of \eqref{head} as desiblack. 
\end{proof}

\end{document}